\documentclass[11pt]{elsarticle}

\usepackage{amssymb,amsthm,amsmath}
\usepackage{graphicx, epstopdf}

\newtheorem{theorem}{Theorem} 
\newtheorem{lem}{Lemma} 

\newtheorem{cor}[theorem]{Corollary}

\numberwithin{subcase}{case}

\numberwithin{subclaim}{claim}

\theoremstyle{definition}

\newtheorem{defn}{Definition}

\begin{document}
\begin{frontmatter}

\title{Forbidden Subgraphs for Chorded Pancyclicity}

\author[spel]{Megan Cream}
\ead{mcream@spelman.edu}
\address[spel]{Department of Mathematics, Spelman College, 350 Spelman Lane SW, Atlanta, GA 30314}

\author[emory]{Ronald J. Gould\fnref{fn1}}
\ead{rg@emory.edu}
\address[emory]{Department of Mathematics and Computer Science,  Emory University, 400 Dowman Drive, Atlanta, GA 30322}

\author[ken]{Victor Larsen\corref{cor1}}
\ead{vlarsen@kennesaw.edu}
\address[ken]{Department of Mathematics, Kennesaw State University, 1100 S. Marietta Parkway, Marietta, GA 30060}

\cortext[cor1]{Corresponding author}
\fntext[fn1]{Supported by Heilbrun Distinguished Emeritus Fellowship}

\begin{abstract}
We call a graph $G$ {\it pancyclic} if it contains at least one cycle of every possible length $m$, for $3\le m\le |V(G)|$.  
In this paper, we define a new property called {\it chorded pancyclicity}.  
We explore forbidden subgraphs in claw-free graphs sufficient to imply that the graph contains at least one chorded cycle of every possible length $4, 5, \ldots, |V(G)|$.  
In particular, certain paths and triangles with pendant paths are forbidden.
\end{abstract}
\begin{keyword}
Pancyclic \sep chorded cycle\sep forbidden subgraph\sep Hamiltonian
\end{keyword}
\end{frontmatter}

\section{Introduction}

In the past, forbidden subgraphs for hamiltonian properties in graphs have been widely studied (for an overview, see \cite{FG}).  
A graph containing a cycle of every possible length from three to the order of the graph is called pancyclic.
The property of pancyclicity is well-studied.
In this paper, we define the notion of chorded pancyclicity, and study forbidden subgraph results for chorded pancyclicity.   
We consider only $K_{1, 3}$-free (or {\it claw}-free) graphs, and we forbid certain paths and triangles with pendant paths.  

Further, we consider only simple claw-free graphs.  
In this paper we let $G$ be a graph and $P_t$ be a path on $t$ vertices.  
Let $Z_i$ be a triangle with a pendant $P_{i}$ adjacent to one of the vertices of the triangle.  
In particular, we will be considering the graphs $Z_1$ and $Z_2$, shown in Figure \ref{fig:z1andz2}.
\begin{figure}[h]
\centering
\includegraphics[width=7.6cm]{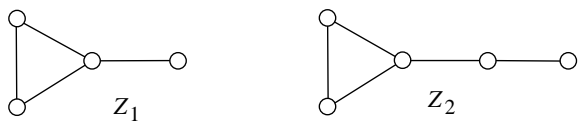}
\caption{The graphs $Z_1$ and $Z_2$}
\label{fig:z1andz2}
\end{figure}
For $S \subseteq V(G)$, let $G[S]$ denote the subgraph of $G$ induced by $S$.
Let $N_G(u)$ denote the set of neighbors of the vertex $u$, that is, the  vertices adjacent to $u$ in the graph $G$.  
Let $N_G[u]=N_G(u)\cup\{u\}$.
A graph is called \emph{traceable} if it contains a hamiltonian path.
We use $H\,\square\, G$ to denote the Cartesian product of $H$ and $G$.
For terms not defined here see \cite{G}.  
We will first note well-known results on forbidden subgraphs for pancyclicity.

 \begin{theorem} Let $R, S$ be connected graphs and let $G$ be a 2-connected graph of order $n\ge10$ such that $G\neq C_n$.  
 Then if $G$ is $\{R,S\}$-free then $G$ is pancyclic for $R=K_{1,3}$ when $S$ is either $P_4, P_5, P_6, Z_1,$ or $Z_2$.  
\label{thm:pan}
 \end{theorem}

The proof of a theorem (Theorem 4) in \cite{GH} yields the following result.

\begin{theorem}{\rm\cite{GH}}
If $G$ is a 2-connected graph of order $n\ge 10$ that contains no induced subgraph isomorphic to $K_{1,3}$ or $Z_1$, then $G$ is either a cycle or $G$ is pancyclic.  
\label{thm:z1}
\end{theorem}

Gould and Jacobson proved a similar result for $Z_2$ in \cite{GJ}.  

\begin{theorem}{\rm\cite{GJ}}
If $G$ is a 2-connected graph of order $n\ge 10$ that contains no induced subgraph isomorphic to $K_{1,3}$ or $Z_2$, then $G$ is either a cycle or $G$ is pancyclic. 
\label{thm:z2} 
\end{theorem}

Faudree, Gould, Ryjacek, and Schiermeyer proved a similar result for certain paths in \cite{FRS}.

\begin{theorem} {\rm\cite{FRS}}
If $G$ is a 2-connected graph of order $n\ge 6$ that is $\{K_{1,3}, P_5\}$-free, then $G$ is either a cycle or $G$ is pancyclic. 
\label{thm:p5} 
\end{theorem}

\begin{theorem} {\rm\cite{FRS}}
If $G$ is a 2-connected graph of order $n\ge 10$ that is $\{K_{1,3}, P_6\}$-free, then $G$ is either a cycle or $G$ is pancyclic. 
\label{thm:p6} 
\end{theorem}

Theorem \ref{thm:p5} implies the following result for $P_4$.   

\begin{theorem} {\rm\cite{FRS}}
If $G$ is a 2-connected graph of order $n\ge 6$ that is $\{K_{1,3}, P_4\}$-free, then $G$ is either a cycle or $G$ is pancyclic. 
\label{thm:p4} 
\end{theorem}

In this paper, we will extend each of these theorems to analogous results on chorded pancyclicity.   

\section{Results}

\begin{defn} 
A graph $G$ of order $n$ is called \emph{chorded pancyclic} if it contains a chorded cycle of every length $m$, $4\le m\le n$.   
\end{defn}

Note first that not all pancyclic graphs are chorded pancyclic.  
The graph in Figure \ref{fig:NonPan} is pancyclic, but contains no chorded 5-cycle. 
Further, the graph in Figure \ref{fig:inffamily} represents an infinite family of pancyclic graphs that do not contain a chorded 4-cycle.   

\begin{figure}[h]
\centering
\begin{minipage}{.4\textwidth}
\centering
\includegraphics[width=3cm]{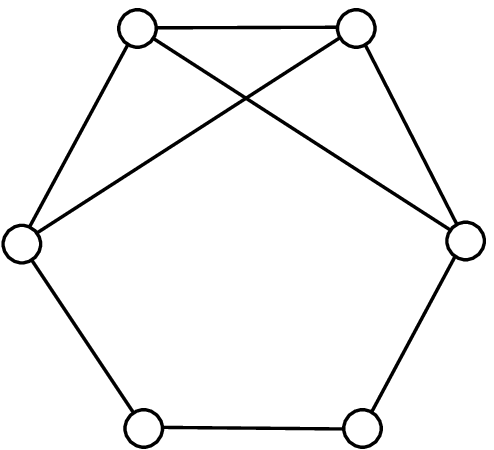}
\caption{ A pancyclic graph with no chorded 5-cycle.}
\label{fig:NonPan}
\end{minipage}
\begin{minipage}{.1\textwidth}
\hspace{1cm}
\end{minipage}
\begin{minipage}{.4\textwidth}
\centering
\vspace{.4cm}
\includegraphics[width=4.1cm]{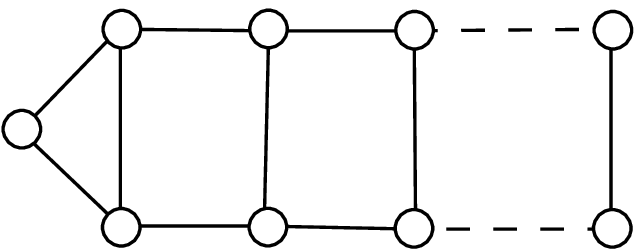}
\vspace{.75cm}
\caption{An infinite family of pancyclic graphs with no chorded $C_4$.}
\label{fig:inffamily}
\end{minipage}
\end{figure}

An important tool used in the proofs of this paper is a $k$-tab, which we define here.
\begin{defn}
Given a subgraph $H$ of $G$, a \emph{$k$-tab on $H$} is a path $a_0a_1\cdots a_{k+1}$ with $k$ internal vertices so that 
$\{a_1,\ldots,a_k\}\subseteq V(G)-V(H)$ and
there exist distinct vertices $u,v\in V(H)$ where $a_0=u$ and $a_{k+1}=v$.
\end{defn}
Note that, because there are no cut-vertices in a 2-connected graph, every proper subgraph must have a $k$-tab for some $k>0$.
Another tool we use in proofs, to better highlight when a subgraph is induced in $G$, is a \emph{frozen set} $F$.  
Depending upon particular cases, the frozen set may be modified during the course of a proof.
If a vertex is in the frozen set, it has no neighbors other than those already given.  
In particular, if a subgraph contains only frozen vertices, then it must be induced.  
Also, if every vertex of $G$ is in $F$, then we have completely described the graph $G$.  
In figures, frozen vertices will be indicated by a solid vertex.

We now turn our attention to extending the results for forbidden subgraphs in Theorem \ref{thm:pan}.
The following simple lemma has appeared many times in the literature.
\begin{lem} Let $G$ be claw-free.  For any $x\in V(G)$, $N_G(x)$ is either connected and traceable, or two disjoint cliques.  
\label{lem:clique}
\end{lem}
Using Lemma \ref{lem:clique} we prove the following useful lemma.
\begin{lem} Let $G$ be a $K_{1,3}$-free graph.  For any $x\in V(G)$ and any integer $k\geq3$, if $\deg_G(x)\ge 2k-1$ then there is a chorded $(k+1)$-cycle in $G$. 
\label{lem:k+1cc}
\end{lem}

\begin{proof}
Consider a vertex $x\in V(G)$ such that $\deg_G(x)=m\ge2k-1$.  
Lemma \ref{lem:clique} implies that $N_G(x)$ is either connected and traceable, or two disjoint cliques.
\vspace{.1in}

\noindent{\it Case 1:} Suppose that $N_G(x)$ is connected and traceable.

Let $v_1 v_2 \cdots v_k v_{k+1}\cdots v_m$ be a hamiltonian path in $N_G(x)$. 
Then $x v_1 v_2 \cdots v_k x$ is a $(k+1)$-cycle in $G$ with chord $xv_2$ (in fact, there are $k-2$ chords in this $(k+1)$-cycle).
\vspace{.1in}

\noindent{\it Case 2:} Suppose that $N_G(x)$ is two disjoint cliques.

Partition the $m$ vertices into two cliques.
Then at least one clique has at least $k$ vertices.
If the vertices  $\{v_1, v_2, \ldots, v_k\}$ form a clique in $N_G(x)$, then $x v_1 v_2\cdots v_k x$ is a $(k+1)$-cycle in $G$ with chord $xv_2$. 
Thus, the lemma is proven.
\end{proof}

\begin{theorem}\label{thm:Z2}
Let $G$ be a 2-connected graph of order $n\ge 10$.  If $G$ is $\{K_{1,3}, Z_2\}$-free, then 
$G= C_n$ or $G$ is chorded pancyclic. 
\end{theorem}

\begin{proof}
Suppose that $G$ is a 2-connected graph of order $n\ge10$ that is $\{K_{1,3},Z_2\}$-free and that $G$ is not $C_n$.
By Theorem \ref{thm:z2}, we know $G$ must be pancyclic.  For the sake of contradiction, suppose that $G$ is not chorded pancyclic.
Let $m$ be the largest value with $4\le m<n$ such that every $m$-cycle in $G$ does not contain a chord.
First we show that $m$ must be 4, and then we show that there is a chorded 4-cycle in $G$.

Suppose that $m\geq5$ and consider an $m$-cycle $C=v_1v_2v_3\cdots v_m v_1$ in $G$.  
Since $G$ is 2-connected and $m<n$, there exists a vertex $x\in V(G)-V(C)$ such that $xv\in E(G)$ for some $v\in V(C)$.  
Without loss of generality, we may assume that $xv_1\in E(G)$.  
Then $\{v_1,x,v_2,v_m\}$ induces a claw in $G$ unless $xv_2$, $xv_m$ or $v_2v_m$ is an edge.
If $v_2v_m$ is added then $C$ is a chorded $m$-cycle and we are done.  
By symmetry, adding either $xv_2$ or $xv_m$ as an edge is equivalent so, without loss of generality, we assume that $xv_2\in E(G)$.  
Now $\{x,v_1,v_2,v_3,v_4\}$ induces a $Z_2$ in $G$.  
The only two edges that can eliminate this induced $Z_2$ without adding a chord to the $m$-cycle $C$ are $xv_3$ and $xv_4$.

If $xv_4\in E(G)$ then $v_1v_2 xv_4v_5 \cdots v_m v_1$ is an $m$-cycle with chord $v_1x$.  
Thus we may assume that $xv_3\in E(G)$.
Now $\{x,v_2,v_3,v_4,v_5\}$ induces $Z_2$ in $G$.   
The only edges that will eliminate this induced $Z_2$, but will not add a chord to $C$ are $xv_4$ and $xv_5$.  
If $xv_4\in E(G)$ then $v_1 v_2 x v_4 v_5 \cdots v_m v_1$ is an $m$-cycle with chord $xv_1$.  
If instead $xv_5\in E(G)$, then $v_1 v_2 v_3 x v_5 \cdots v_m v_1$ is an $m$-cycle with chord $xv_1$.  

Therefore, it follows that $m=4$.
By Lemma \ref{lem:k+1cc}, it follows that $\Delta(G)\leq 4$.
Let $D=vwxyzv$ be a 5-cycle in $G$ with chord $vx$.

Suppose that both $v$ and $x$ can be added to $F$; that is, suppose that $v$ and $x$ have degree 3 in $G$.
Because $G$ is 2-connected and $n>5$, $D$ has a $k$-tab for some $k$ with two distinct endpoints in $\{w,z,y\}$.
By symmetry, we may assume that $y$ is an endpoint of the $k$-tab and thus $ay$ is an edge for $a\in V(G)-V(D)$.
But now $\{y,a,x,z\}$ induces a claw, unless $za$ is an edge (recall that $x$ is frozen).
Further, $\{w,v,x,y,a\}$ induces a $Z_2$ unless one of $wy$ or $wa$ is an edge.
Because the edge $wy$ gives a 4-cycle $wyxvw$ in $G$ with chord $wx$, it follows that $za$ and $wa$ are both edges.
But now $G[V(D)\cup\{a\}]$ is isomorphic to $K_3\,\square\, K_2$, which is vertex-transitive.
Therefore, all vertices of $G$ can be added to $F$, as vertices can be relabeled so that any edge from $V(D)\cup\{a\}$ to $V(G)-(V(D)\cup\{a\})$ is a new edge using the vertex $x$.
As a side note, this symmetry argument is used many times in further proofs to add vertices to $F$.

Thus it follows that, without loss of generality, $ax$ is an edge for $a\in V(G)-V(D)$.
Now $\{x,a,w,y\}$ induces a claw unless one of $\{wy,wa,ay\}$ is an edge.
Because $wy$ gives a 4-cycle $wyxvw$  with chord $wx$ and $wa$ gives a 4-cycle $waxvw$ with chord $wx$, it must be the case that $ay$ is an edge.
Note that $\deg_G(x)=4$ so we may add $x$ to $F$ (now $F=\{x\}$).

Because $n>6$ we look for a neighbor for a vertex $b\in V(G)-(V(D)\cup\{a\})$.
Suppose that both $v$ and $y$ can be added to $F$.
Then $z$ must also be added to $F$, as the edge $bz$ would give an induced claw $\{z,b,v,y\}$.
Now $w$ and $a$ are the only unfrozen vertices, so we may assume by symmetry that $ba$ is an edge.
Now $\{w,v,x,a,b\}$ induces a $Z_2$ unless $wb$ is an edge, as $wa$ has already been ruled out and $v$ and $x$ are frozen by supposition.
If $wb$ is an edge, then note that swapping the labels of $v,z,y$ with $w,b,a$, respectively, does not change the graph. 
Thus an additional edge to $w,b,$ or $a$ is equivalent to an additional edge to $v,z,$ or $y$ which are all frozen.
Therefore all vertices can be added to $F$. This implies that $n=7$, which is a contradiction.

Thus it follows that, without loss of generality, $by$ is an edge.
To avoid an induced claw on $\{y,b,z,x\}$ we must have the edge $bz$.
Now $\{w,v,x,y,b\}$ induces a $Z_2$.
Because the edges $yw$ and $yv$ create chorded 4-cycles on $\{w,v,x,y\}$ and the edge $bv$ would give the 4-cycle $bvzyb$ with chord $bz$, it follows that $bw$ must be an edge to eliminate this induced $Z_2$.
Note that $\deg_G(y)=4$, so we also add $y$ to $F$ (now $F=\{x,y\}$).

Let $R=V(D)\cup\{a,b\}$.  Because $G$ is 2-connected and $n>7$, there exists a vertex $c\in V(G)-R$ with a neighbor in $R-\{a\}$.  
Because $x$ and $y$ are frozen, and $G[R-\{a\}]$ is isomorphic to $K_3\,\square\, K_2$, the edges $cw,cv,cz,$ and $cb$ are all equivalent.
We assume, without loss of generality, that $cz$ is an edge (see Figure \ref{fig:Z2Conj1}).

\begin{figure}[h]
\centering
\begin{minipage}{.4\textwidth}
\centering
\includegraphics[width=3cm]{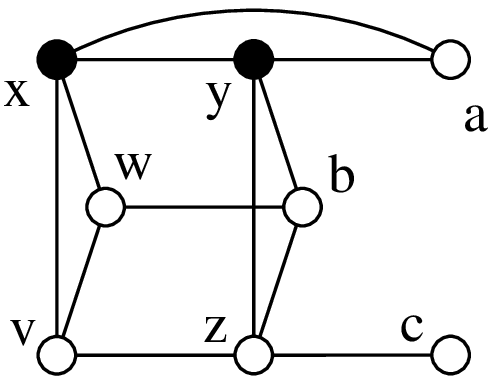}
\caption{Adding $c$ to $G[R]$ in the proof of Theorem \ref{thm:Z2}.}
\label{fig:Z2Conj1}
\end{minipage}
\begin{minipage}{.1\textwidth}
\hspace{1cm}
\end{minipage}
\begin{minipage}{.4\textwidth}
\centering
\includegraphics[width=3cm]{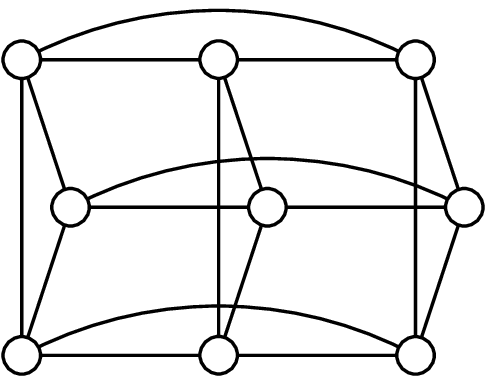}
\caption{$K_3\,\square\, K_3$ contains no chorded $C_4$. }
\label{fig:maxdeg4}
\end{minipage}
\end{figure}

Now to prevent $\{z,c,v,y\}$ from inducing a claw we must have the edge $cv$.
Because the degree of $v$ and $z$ are now 4, we add them to $F$ (so $F=\{x,y,v,z\}$).
To prevent $\{c,v,z,x,a\}$ from inducing a $Z_2$ we must have the edge $ca$ as these are the only unfrozen vertices.

Because $n>8$ there exists another vertex $d\in V(G)-(R\cup\{c\})$.
By symmetry any edge from $d$ to one of the unfrozen vertices $\{w,a,b,c\}$ is equivalent so we may assume without loss of generality that $dw$ is an edge.
To prevent $\{w,x,d,b\}$ from inducing a claw $db$ must be an edge.
Now to prevent $\{w,b,d,y,a\}$ from inducing a $Z_2$ we must have the edge $da$ because $wa$ give a 4-cycle $waxvw$ with chord $wx$ and $ba$ gives a 4-cycle $bayzb$ with chord $by$.
Now that $da$ is an edge, we must also have the edge $dc$ to prevent $\{a,d,c,x\}$ from inducing a claw.
But now $G$ is isomorphic to $K_3\,\square\, K_3$ (see Figure \ref{fig:maxdeg4}), all vertices have degree 4 and are thus frozen, and $n=9$.  This is a contradiction, so the theorem is proven.
\end{proof}

As $Z_1$ is an induced subgraph of $Z_2$, Theorem \ref{thm:z1} implies the following corollary.

\begin{cor} Let $G$ be a 2-connected graph of order $n\ge 10$.  If $G$ is $\{K_{1,3}, Z_1\}$-free, then $G=C_n$ or $G$ is chorded pancyclic. 
\end{cor}

We now turn our attention to forbidden paths; many of these theorems will be proven using a collection of lemmas.

\begin{theorem} Let $G$ be a 2-connected graph of order $n\ge 5$.  If $G$ is $\{K_{1,3}, P_4\}$-free, then $G$ is chorded pancyclic. 
\label{thm:p4ch}
\end{theorem}

\begin{proof}
Let $G$ be a 2-connected graph of order $n\ge5$ that is $\{K_{1,3},P_5\}$-free.
One can easily verify that $G$ is chorded pancyclic if $n=5$.
Therefore we assume that $n\geq6$.
By Theorem \ref{thm:p4}, we know that $G$ is pancyclic.  
Suppose, for the sake of contradiction, that $G$ is not chorded pancyclic. 
 Let $m$ be the largest number with $4\le m\le n$ such that every $m$-cycle in $G$ is chordless.
 Any chordless $m$-cycle for $m>4$ contains an induced $P_4$, so it follows that $m=4$.

Consider a 4-cycle $C=v_1 v_2 v_3 v_4 v_1$ in $G$.  
Since $n\ge 5$ and $G$ is 2-connected, there exists a vertex $x\not \in V(C)$ such that $xv\in E(G)$ for some $v\in V(G)$.  
Without loss of generality, we may assume that $v_1 x\in E(G)$.  
To avoid an induced claw on $\{v_1,x,v_2,v_4\}$, one of $v_2v_4$, $xv_4$, or $xv_2$ must be in $E(G)$.
If $v_2v_4$ is an edge, then $C$ is a chorded 4-cycle, which is a contradiction.

Then by symmetry, we may assume without loss of generality that $xv_2$ is an edge.
Because $x v_1 v_4 v_3$ cannot be an induced $P_4$ subgraph of $G$, $G$ must contain either $xv_4$, $xv_3$, or $v_1v_3$ as an edge. 
The edge $v_1 v_3$ creates a chord of the 4-cycle $C$, a contradiction.  
The edge $v_3x$ yields the 4-cycle $v_1 x v_3 v_2 v_1$ with chord $v_2x$, a contradiction. 
Similarly, the edge $v_4x$ yields the 4-cycle $v_1 v_2 x v_4 v_1$ with chord $v_1x$, again a contradiction.
Thus, every 4-cycle in $G$ must be chorded.
\end{proof}

\begin{lem}\label{lem:ChordedC5}
Let $G$ be a 2-connected graph of order $n\ge 8$.  If $G$ is $\{K_{1,3}, P_5\}$-free and $G$ contains a $C_4$, then $G$ contains a chorded $C_5$.
\end{lem}
\begin{proof}
Let $G$ be a 2-connected graph of order $n\ge 8$ that is $\{K_{1,3}, P_5\}$-free, and
let $C=wxyzw$ be a $4$-cycle in $G$.
Suppose, for the sake of contradiction, that $G$ does not have a chorded $C_5$.
Because $n>4$ and $G$ is 2-connected, there is a $k$-tab on the cycle $C$.  
Choose a $k$-tab $Q=a_0\cdots a_{k+1}$  which minimizes $k$.
Over all minimal $k$-tabs, choose one where $a_0a_{k+1}$ is an edge of $C$ if possible.
Without loss of generality, we may assume that $x=a_0$.

Suppose $Q$ is a 1-tab.  
If $a_2$ is $y$ (or $w$ by symmetry), then $xa_1yzwx$ is a 5-cycle with chord $xy$.
If $a_2=z$ then $\{x,a_1,y,w\}$ cannot be an induced claw.  But $a_1w$ and $a_1y$ contradict our choice of $Q$, so it follows that $wy$ is an edge.
Now $wyxa_1zw$ is a 5-cycle with chord $xw$.

Suppose that $k\geq 3$.
Then let $a_0'$ be a neighbor of $x=a_0$ on $C$ which is not $a_{k+1}$.
Now $a'_0xa_1a_2a_3$ is an induced $P_5$ unless there is another edge among $\{a'_0,x,a_1,a_2,a_3\}$.
However, each of these edges creates either a 1-tab or a 2-tab on $C$, which contradicts the minimality of $Q$.

Therefore, it follows that $k=2$.
If $a_3$ is $y$ (or $w$ by symmetry), then to avoid an induced claw on $\{x,a_1,y,w\}$ one of $wy,ya_1,wa_1$ must be an edge.
However, each of these create a chorded 5-cycle.
Thus $a_3=z$.
To prevent $\{x,a_1,y,w\}$ from inducing a claw without creating a 1-tab, we must have the edge $yw$.
One can check that any further edge among $R=\{x,y,w,z,a_1,a_2\}$ creates a chorded 5-cycle.
However, $n>6$, so there exists some $b\in V(G)-R$ that is adjacent to a vertex in $R$.

If $by$ is an edge (or $bw$ by symmetry) then to avoid an induced claw $\{y,b,x,z\}$, either $bx$, $bz$, or $xz$ is an edge.
All three give chorded 5-cycles, so we may add $y$ and $w$ to the frozen set $F$.
Now $F=\{y,w\}$.

Suppose that we can now add both $x$ and $z$ to $F$ so that $F=\{y,w,x,z\}$.
Then, without loss of generality, we may assume that $ba_1$ is an edge.  
But now $ba_1xwz$ is an induced $P_5$ because $x,w,z$ are all frozen.
This is a contradiction, and thus at least one of $\{x,z\}$ cannot be added to $F$.
In particular, let $bx$ be an edge.
To avoid an induced claw $\{x,b,a_1,w\}$ it follows that $ba_1$ is an edge.
Also, $bxwza_2$ cannot be an induced $P_5$, so there is another edge among these vertices.
Because $bz$ creates a shorter $k$-tab on $C$, it follows that $ba_2$ is an edge.
Let $R'=R\cup\{b\}$ and note that any additional edge in $G[R']$ creates a chorded 5-cycle.
By symmetry with $\{y,w\}$ we can now add $b$ and $a_1$ to $F$ (see Figure \ref{fig:P5C5*}).

\begin{figure}[h]
\centering
\includegraphics[width=3cm]{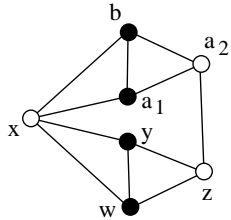}
\caption{The graph $G[R']$ from Lemma \ref{lem:ChordedC5}.}
\label{fig:P5C5*}
\end{figure}

However, $n>7$, so there exists some $c\in V(G)-R'$ that is adjacent to a vertex in $R'$.
If $cx$ is an edge for $c\in V(G)-R'$ then $\{x,c,b,w\}$ induces a claw as $b$ and $w$ are frozen.
Therefore, we can also add $x$ to $F$ and, without loss of generality, $cz$ is an edge.
But now $czwxb$ is an induced $P_5$, which is a contradiction.
\end{proof}

\begin{lem}\label{lem:ChordedC4}
Let $G$ be a 2-connected graph of order $n\ge 7$.  If $G$ is $\{K_{1,3}, P_5\}$-free and $G$ contains a chorded $C_5$, then it also contains a chorded $C_4$.
\end{lem}
\begin{proof}
Let $G$ be a 2-connected graph of order $n\ge 7$ that is $\{K_{1,3}, P_5\}$-free, and let $C=vwxyzv$ be a 5-cycle with chord $vx$.
Suppose, for the sake of contradiction, that $G$ has no chorded $C_4$.

We show that both $x$ and $v$ can be added to $F$.
If $ax\in E(G)$ for some $a\in V(G)-V(C)$, then to avoid a claw there must be some edge among $\{w,a,y\}$.
The edge $aw$ gives a 4-cycle $vwaxv$ with chord $wx$ and the edge $wy$ gives a 4-cycle $wyxvw$ with chord $wx$.
Thus $ay$ must be an edge of $G$.
Any further edge among $\{v,w,x,y,z,a\}$ yields a chorded $4$-cycle, but such an edge must exist because otherwise $ayzvw$ is an induced $P_5$.
Therefore, we conclude that $x$ (and $v$ by symmetry) can be added to $F$.

Because $G$ is 2-connected, there is a tab on $C$.
Therefore we may assume that $ay$ (or $az$ by symmetry) is an edge for some $a\in V(G)-V(C)$.
To avoid an induced claw on $\{y,a,x,z\}$, we must have the edge $az$.
By symmetry with $\{v,x\}$, both $z$ and $y$ can now be added to $F$.
Because $n>6$, there is a vertex $b$ adjacent to either $a$ or $w$.
If $ba$ is an edge then $bazvx$ is an induced $P_5$, and if $bw$ is an edge then $bwvzy$ is an induced $P_5$.
In either case we arrive at a contradiction, so we have proven the lemma.
\end{proof}

\begin{theorem} Let $G$ be a 2-connected graph of order $n\ge 8$.  If $G$ is $\{K_{1,3}, P_5\}$-free, then $G$ is chorded pancyclic. 
\label{thm:p5ch}
\end{theorem}

\begin{proof}
From Theorem \ref{thm:p5} we know $G$ must be pancyclic.  
Then by applying Lemma \ref{lem:ChordedC5} followed by Lemma \ref{lem:ChordedC4}, we find a chorded $m$-cycle in $G$ for $m=4,5$.
Any chordless $m$-cycle for $m>5$ contains an induced $P_5$.
Therefore $G$ contains a chorded $m$-cycle for $4\leq m\leq n$.
\end{proof}

Note that the graph in Figure \ref{fig:P5C5*} shows that Theorem \ref{thm:p5ch} is sharp.
We now turn to $P_6$ as a forbidden subgraph.  We will prove our result using three lemmas.


\begin{lem}\label{lem:P6C5*}
Let $G$ be a 2-connected graph of order $n\ge 11$.  If $G$ is $\{K_{1,3}, P_6\}$-free and $G$ contains a $C_4$, then $G$ contains a chorded $C_5$.
\end{lem}
\begin{proof}
Let $G$ be a 2-connected graph of order $n\ge 11$ that is $\{K_{1,3}, P_6\}$-free, and let $C=vwxyv$ be a 4-cycle in $G$.
Suppose, for the sake of contradiction, that $G$ does not have a chorded $C_5$.
Because $G$ is 2-connected and $n>5$, there is a $k$-tab on the cycle $C$.  
We choose a $k$-tab $Q=a_0a_1\cdots a_{k+1}$ that minimizes $k$.  
Over all minimal $k$-tabs, choose one where $a_0a_{k+1}$ is an edge of $C$ if possible.
If $k\geq4$, then let $a'_0$ be a neighbor of $a_0$ on $C$ which is not $a_{k+1}$ and now $a'_0a_0a_1\cdots a_{4}$ is an induced $P_6$.  Therefore we may assume that $1\leq k\leq 3$.
\vspace{.1in}

\noindent{\it Case 1.} Suppose that $Q$ is a 1-tab.

If $a_0a_2$ is an edge of $C$ then $G[\{v,w,x,y,a_1]$ contains a chorded 5-cycle.  
Otherwise, we may assume $a_0=v$ and $a_2=x$.  
To avoid a claw centered at $v$ or a 1-tab with endpoints adjacent in $C$ it follows that $wy$ is an edge.  
Now $va_1xywv$ is a 5-cycle with chord $wx$.
\vspace{.1in}

\noindent{\it Case 2.} Suppose that $Q$ is a 2-tab.

First, we show that $a_0a_3$ cannot be an edge in $C$.
If $a_0a_3$ is an edge, we may assume it is $vw$.
To avoid a claw induced by $\{v,w,y,a_1\}$ or a 1-tab on $C$, we must have $wy$ as an edge.  
But now $va_1a_2wyv$ is a 5-cycle with chord $vw$.
Therefore we will assume that $v=a_0$ and $x=a_3$.

Note that $wy$ must be an edge to avoid an induced claw $\{v,w,y,a_1\}$ or a 1-tab on $C$.  
If $w$ (or $y$ by symmetry) has another neighbor $z\in V(G)-V(C)$ then to avoid an induced claw $\{w,z,v,x\}$ or a 1-tab on $C$ we must have the edge $vx$.  
But now $va_1a_2xwv$ is a 5-cycle with chord $vx$. 
Therefore, $w$ and $y$ are added to $F$ for the remainder of this case.
\vspace{.1in}

\noindent{\it Subcase 2.1.} Suppose that $v$ and $x$ can be added to $F$ so that $F=\{w,y,v,x\}$.

Because $G$ is 2-connected and $n\geq8$, there exist distinct vertices $b_1$ and $b_2$ so that $b_1a_1$ and $b_2a_2$ are edges.
To avoid an induced claw on either $\{a_1,b_1,v,a_2\}$ or $\{a_2,b_2,v,a_1\}$ we must also have the edges $b_1a_2$ and $b_2a_1$.
Also $\{a_1,b_1,b_2,v\}$ cannot be an induced claw so $b_1b_2$ is an edge.

Because $n>8$ there is another vertex $c$.  
If $ca_1$ (or $ca_2$ by symmetry) is an edge, then either $\{a_1,c,b_1,v\}$ is an induced claw or $a_1cb_1b_2a_2a_1$ is a 5-cycle with chord $a_1b_1$.
Thus, without loss of generality we let $cb_1$ be an edge.  
But now $cba_1vyx$ is an induced $P_6$, which is a contradiction.
\vspace{.1in}

\noindent{\it Subcase 2.2.} Suppose that at least one of $v$ and $x$ cannot be added to $F$. 

In particular, we will assume that $bx$ is an edge for a vertex $b\notin V(C)\cup V(Q)$.
The current state of $G$ (except for the edge $ba_2$ which we show next) is given in Figure \ref{fig:P6Case2}.
Let $R=\{v,w,x,y,a_1,a_2,b\}$.
Because $w,y\in F$, the neighborhood of $x$ is not traceable and, by Lemma \ref{lem:clique}, $N_G(x)$ must be two disjoint cliques.
Therefore $ba_2$ is an edge.

First, we show that $ba_1$ cannot be an edge.
If $ba_1$ is an edge, then $G[V(C)]$ and $G[\{x,a_2,b,a_1\}]$ are both isomorhpic to $K_4-e$.
By symmetry, we may add $a_2$ and $b$ to $F$.
Because $N_G(x)$ is two disjoint cliques, we may also add $x$ to $F$.
Because $G$ is 2-connected and $n\geq9$, there exist distinct vertices $c,c'\in V(G)-R$ so that $cv$ and $c'a_1$ are edges.
To avoid the induced claw $\{a_1,c',v,b\}$ we also have the edge $c'v$.
By Lemma \ref{lem:clique} the neighborhood of $v$ is two disjoint cliques, as it is not traceable.
Thus $cc'$ and $ca_1$ are also edges.
If $N_G(v)$ contains a $K_4$, then $G$ has a chorded 5-cycle, so $v$ must now be added to $F$.
By symmetry, $a_1$ is also added to $F$.
If $G$ has another vertex $d$ then either $dcvyxb$ or $dc'vyxb$ is an induced $P_6$, so $n=9$ which is a contradiction.
Therefore $ba_1$ is not an edge, and to avoid any chorded 5-cycle it follows that the induced subgraph $G[R]$ is exactly the graph shown in Figure \ref{fig:P6Case2}.

\begin{figure}[h]
\centering
\includegraphics[width=3.6cm]{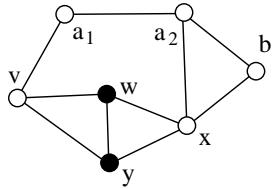}
\caption{The graph $G[R]$ is the starting point for Subcase 2.2 of Lemma \ref{lem:P6C5*}.}
\label{fig:P6Case2}
\end{figure}

Now we show that $b$ cannot be added to $F$ yet.
Suppose that $b$ is added to $F$, so that $F=\{w,y,b\}$.
Because $N_G(x)$ is two disjoint cliques, we may also add $x$ to $F$.
Either $v$ can now be added to $F$, or $v$ has another neighbor $c$.
If $vc$ is an edge then, because $N_G(v)$ is two disjoint cliques, $ca_1$ is an edge.
By symmetry with $b,x$ we can add $c$ and $v$ to $F$.
However, regardless of whether $v$ has degree 3 or 4 in $G$, since $G$ is 2-connected and $n>8$ we may assume that $a_1$ has a neighbor $d\notin R$.
But now $da_1vyxb$ is an induced $P_6$, which is a contradiction.

Because $b$ cannot be added to $F$, it follows that $b$ has a neighbor $b'\in V(G)-R$.
Now $b'bxyva_1$ cannot induce a $P_6$ so $b'$ has another neighbor among these vertices.
Note that $b'a_1$ creates a 5-cycle $b'a_1a_2xbb'$ with chord $ba_2$ so this cannot be an edge.
Thus, the edge $b'v$ would create an induced claw $\{v,b',a_1,y\}$, and it follows that $b'x$ must be an edge.
Because $N_G(x)$ is two disjoint cliques, $b'a_2$ is an edge.
We can now add $x$ to $F$, because otherwise one of the cliques in $N_G(x)$ would be a $K_4$ which gives a chorded $5$-cycle in $N_G[x]$.
If $b$ has a neighbor $c\notin R\cup\{b'\}$, then to prevent $cbxyva_1$ being an induced $P_6$ either $cv$ or $ca_1$ is an edge.
When $ca_1$ is an edge then $ca_1a_2b'bc$ is a 5-cycle with chord $a_2b$, and when $ca_1$ is not an edge then $cv$ is an edge and $\{v,c,a_1,y\}$ induces a claw.
Therefore $b$ (and $b'$ by symmetry) are also added to $F$.

Now the only unfrozen vertices are $v,a_1,$ and $a_2$.
We claim that $a_2$ can also be added to $F$.
Suppose instead that $a_2$ has a  neighbor $c\notin R\cup\{b'\}$.
Then $ca_1$  is also an edge to prevent $\{a_2,c,a_1,x\}$ from inducing a claw. 
Also, $ca_1vyxb$ is not an induced $P_6$ so $cv$ is an edge.
Now $G[\{v,a_1,a_2,c\}]$ is isomorphic to $G[V(C)]$, so by symmetry $\{c,a_1,a_2\}$ are added to $F$.
Because $G$ has no cut vertex, $v$ is also added to $F$.
All vertices are frozen, so it follows that $n=9$ which is a contradiction.

Therefore $a_2$ is added to $F$, as seen in Figure \ref{fig:P6SubSubCase}.  
Because $G$ is 2-connected and $n\geq10$, there exist distinct vertices $c$ and $c'$ so that $cv$ and $c'a_1$ are edges.
To avoid the induced claw $\{a_1,c',v,a_2\}$ we also have the edge $c'v$.
But $N_G(v)$ is two disjoint cliques, so $cc'$ and $ca_1$ are also edges.
By symmetry with $G[x,b,b',a_2]$, the vertices $v,c,c',a_1$ are added to $F$.
Thus $F=V(G)$, and it follows that $n=10$ which is a contradiction.

\begin{figure}[h]
\centering
\includegraphics[width=1.6in]{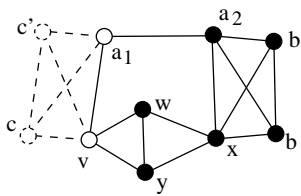}
\caption{The graph resulting from Subcase 2.2 in Lemma \ref{lem:P6C5*}. }
\label{fig:P6SubSubCase}
\end{figure}

\vspace{.1in}

\noindent{\it Case 3.} Suppose that $Q$ is a 3-tab.

Recall that $C=vwxyv$, and that no vertices are in $F$ yet.
Let $R=\{v,w,x,y,a_1,a_2,a_3\}$.
Note that no 4-cycle can have a 1-tab or a 2-tab, or we reduce to an earlier case.
\vspace{.1in}

\noindent{\it Subcase 3.1.} Suppose that $a_0a_4$ is an edge of $C$.  

In particular, we may assume that $a_0=v$ and $a_4=w$.
To prevent $\{v,w,y,a_1\}$ and $\{w,v,x,a_3\}$ from inducing claws without introducing a chorded 5-cycle, $vx$ and $wy$ must be edges.
Note that if $R$  induces any other edge in $G$ then there is a 1-tab or 2-tab on $C$, which is a contradiction.
We add $y$ (and $x$ by symmetry) to $F$ by the following argument.
If $by$ is an edge, then to prevent $byva_1a_2a_3$ being an induced $P_6$ there must be another neighbor of $b$.  
Because $bv$, $ba_1$, and $ba_3$ all create 1- or 2-tabs it follows that only $ba_2$ can be an edge.
But then $\{a_2,a_1,a_3,b\}$ induces a claw, because $a_1a_3$ would make a 2-tab on $C$.

We now show that $v$ and $w$ can also be added to $F$.
If $v$ has a neighbor $b\in V(G)-R$ then, to prevent the induced claw $\{v,b,a_1,y\}$, $ba_1$ is also an edge.
To prevent $ba_1a_2a_3wx$ being an induced $P_6$ there must be another neighbor of $b$ among those vertices.  
However, $bw$ and $ba_3$ create 1- and 2-tabs on $C$.
Therefore $ba_2$ is an edge, but now $a_2a_3wv$ is a 2-tab on the 4-cycle $a_2a_1vba_2$, which is a contradiction.
Therefore $F=\{x,y,v,w\}$.

Because $G$ is 2-connected and $n>7$ there is a tab on $G[R]$.
Because only $a_1,a_2,a_3$ are unfrozen we may assume that $ba_1$ (or $ba_3$ by symmetry) is an edge for some $b\in V(G)-R$.
To prevent an induced claw on $\{a_1,b,a_2,v\}$, we must also have the edge $ba_2$.
If $ba_3$ were also an edge then $a_1vwa_3$ would be a 2-tab on the 4-cycle $a_1a_2a_3ba_1$, so $ba_3$ cannot be an edge.
But now $a_1$ can be added to $F$ by the following argument.
If $a_1$ has another neighbor $c\in V(G)-R$ distinct from $b$ then by the same reasoning, we must have the edge $ca_2$ and we cannot have the edge $ca_3$.
By Lemma \ref{lem:clique}, the neighborhood of $a_1$ is two disjoint cliques so $bc$ is also an edge.
But now $G[\{a_1,a_2,b,c\}]$ is isomorphic to $G[V(C)]$ and by symmetry all of $a_1,a_2,b,c$ are added to $F$.
Because $a_3$ is not a cut-vertex in $G$ it follows that $n=9$, which is a contradiction.
Therefore, once we have the edge $ba_1$ we can immediately add $a_1$ to $F$.
Hence $F=\{x,y,v,w,a_1\}$.

Now $b,a_2,$ and $a_3$ are the only unfrozen vertices.
Suppose that $a_3$ has a neighbor $c\notin R\cup\{b\}$.
Then if $cb$ is not an edge then $ca_3wva_1b$ induces a $P_6$, and if $cb$ is an edge then $cba_1a_2a_3c$ is a 5-cycle with chord $ba_2$.
Both are contradictions, so $\deg_G(a_3)=2$ and $a_3$ is added to $F$.
Because $n>8$ there must be another vertex $c$, and now $b$ and $a_2$ are the only unfrozen vertices.
However, if $cb$ is an edge then $cba_1vwa_3$ is an induced $P_6$ and if $ca_2$ is an edge then $\{a_2,c,a_1,a_3\}$ induces a claw.
Both are contradictions, so the lemma is proven in this case.
\vspace{.1in}

\noindent{\it Subcase 3.2.} Suppose that $a_0$ and $a_4$ are not adjacent in $C$.  

In particular, let $a_0=v$ and $a_4=x$.
Recall that $F$ is empty for now.
To prevent $\{v,a_1,w,y\}$ from inducing a claw, and to avoid chorded 5-cycles, $wy$ must be an edge.
Note that now $vx$ cannot be an edge, or we reduce to Subcase 3.1.
Therefore $w$ (and $y$ by symmetry) cannot have any neighbor in $V(G)-R$ without inducing a claw or creating a chorded 5-cycle.
Thus $w$ and $y$ are added to $F$.

Now we show that at least one element of $\{x,v\}$ cannot be added to $F$ yet.
Suppose instead that $F=\{w,y,x,v\}$.
Then, because $G$ is 2-connected, there is a tab on $G[R]$ and we may assume that $a_1$ (or $a_3$ by symmetry) has a neighbor $b\in V(G)-R$.
To prevent the induced claw $\{a_1,b,a_2,v\}$ we must also have the edge $ba_2$.  
To prevent $ba_1vyxa_3$ from being an induced $P_6$ we must have the edge $ba_3$.  
But now $G[\{a_1,a_2,a_3,b\}]$ is isomorphic to $G[\{v,w,x,y\}]$, so by symmetry all vertices are added to $F$.
This implies that $n=8$, which is a contradiction.
Therefore, at least one of $x$ or $v$ has a neighbor in $V(G)-R$.

Suppose, without loss of generality, that $x$ has a neighbor $b\in V(G)-R$.
To prevent the induced claw $\{x,b,a_3,y\}$ we must also have the edge $ba_3$.  
Also, $bxyva_1a_2$ cannot be an induced $P_6$, so $b$ must have another neighbor among these vertices.
The edges $bv$ and $ba_1$ create shorter tabs on $C$, so it follows that $ba_2$ must be an edge (see Figure \ref{fig:P63Tabs}).

\begin{figure}[h]
\centering
\includegraphics[width=3.2cm]{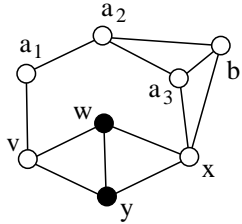}
\caption{A graph for Subcase 3.2 of Lemma \ref{lem:P6C5*}.}
\label{fig:P63Tabs}
\end{figure}

By symmetry with $w$ and $y$, the vertices $a_3$ and $b$ are also added to $F$, so $F=\{w,y,a_3,b\}$.
Because the neighborhood of $x$ is two disjoint frozen cliques, $x$ can also be added to $F$ by Lemma \ref{lem:clique}.
Because $n\geq9$, there is a tab with endpoints in $\{v,a_1,a_2\}$.
Without loss of generality, we may assume that $cv$ is an edge.
Now $ca_1$ must also be an edge to prevent $\{v,c,a_1,y\}$ inducing a claw, and $ca_2$ must be an edge to prevent $cvyxba_2$ from being an induced $P_6$.
But now $G[\{a_1,a_2,v,c\}]$ is isomorphic to $G[V(C)]$ and all vertices can be added to $F$.
Therefore $n=9$, which is a contradiction.
\end{proof}


\begin{lem}\label{lem:P6C4*}
Let $G$ be a 2-connected graph of order $n\ge 10$.  If $G$ is $\{K_{1,3}, P_6\}$-free and $G$ contains a chorded $C_5$, then it also contains a chorded $C_4$.
\end{lem}
\begin{proof}
Let $G$ be a 2-connected graph of order $n\ge 10$ that is $\{K_{1,3}, P_6\}$-free, and let $C=vwxyzv$ be a 5-cycle with chord $vx$.  
Suppose, for the sake of contradiction, that $G$ does not have a chorded $C_4$.
Then there are no additional edges in $G[V(C)]$ and,
by Lemma \ref{lem:k+1cc}, $\Delta(G)\leq4$.
Because $n>5$, there is a vertex $a\in V(G)-V(C)$ adjacent to a vertex of $C$.
\vspace{.1in}

\noindent{\it Case 1.} Suppose that $F=\{x,v\}$.

Because there is a $k$-tab on $C$, we may assume without loss of generality that $ay$ is an edge.
To avoid an induced claw on $\{y,x,z,a\}$ it follows that $az$ must also be an edge of $G$.
By symmetry we may now add $z$ and $y$ to $F$.
Because $n>7$, to avoid a cut-vertex in $G$ there must be two distinct vertices $b$ and $c$ where $ab$ and $wc$ are edges in $G$.

If $aw$ is an edge, then to avoid a claw on $\{a,w,b,y\}$, we must also have the edge $wb$.
But then $\deg_G(w)\geq5$ which condtradicts $\Delta(G)\leq 4$.  
Therefore, $aw$ is not an edge in $G$.

If $ac$ is an edge, then to avoid a claw on $\{a,b,c,z\}$, we must also have the edge $bc$.
Now $bcwxyz$ is an induced $P_6$ unless $bw$ is an edge.  
However, if $bw$ is an edge then $bwcab$ is a 4-cycle with chord $bc$.
Therefore $ac$ (and by symmetry $bw$) is not an edge in $G$.

Now $bayxwc$ is an induced $P_6$ if $bc$ is not an edge, and $cbayxv$ is an induced $P_6$ if $bc$ is an edge.
\vspace{.1in}

\noindent{\it Case 2.} Suppose that at least one of $v$ and $x$ cannot be added to $F$.

In particular, we will assume without loss of generality that $ax$ is an edge in $G$.
To avoid chorded 4-cycles or $\{x,y,a,w\}$ inducing a claw, $ay$ must also be an edge in $G$.
Because $\Delta(G)\leq4$ we can now let $F=\{x\}$.
Let  $R=\{a,v,w,x,y,z\}$ and note that any additional edge in $G[R]$ creates a chorded $C_4$.
Because $n>6$ there must be a vertex $b\in V(G)-R$ adjacent to a vertex of $R$.

We cannot add both $v$ and $y$ to $F$ by the following argument. 
If $v$ and $y$ are added to $F$ then $z$ must also be added to $F$, as the edge $zb$ would create an induced claw on $\{z,b,y,v\}$.
Then only $w$ and $a$ are unfrozen, so we may assume that $ba$ is an edge.
Now $bayzvw$ is an induced $P_6$ unless $bw$ is also an edge.
By symmetry with $\{v,y,z\}$, the vertices $a,w,b$ are added to $F$.
Thus $|V(G)|=7$ which is a contradiction.

Therefore we may assume, without loss of generality, that $by$ is an edge. 
To avoid $\{y,b,a,z\}$ inducing a claw or a chorded $C_4$, $bz$ must also be an edge.
Because $\Delta(G)\leq4$, we add $y$ to $F$ so that $F=\{x,y\}$.
Because $n>7$ there is another vertex $c$ adjacent to a vertex in $R\cup\{b\}$.
\vspace{.1in}

\noindent{\it Subcase 2.1.} Suppose that at least one of $z$ and $v$ cannot be added to $F$.

Then we may assume, without loss of generality, that $cz$ is an edge.
To avoid $\{z,c,b,v\}$ being a claw or a chorded $C_4$ on $\{z,c,b,y\}$, $cv$ must also be an edge.
By symmetry, we can add $z$ and $v$ to $F$ (see Figure \ref{fig:4flower}).
Because $n>8$, there is an unpictured vertex $d$ and, without loss of generality, we assume that $cd$ is an edge.

Since $G$ does not contain a chorded $C_4$, $cw$ is not an edge.
But then $dw$ is an edge, because otherwise $dczyxw$ is an induced $P_6$.
Similarly, $db$ must be an edge to avoid $dcvxyb$ being an induced $P_6$.
Now $ad$ is also an edge, because otherwise $axvzbd$ is an induced $P_6$.
We add $d$ to $F$ because $\Delta(G)\leq4$.
Thus $F=\{x,y,z,v,d\}$.
Now $\{d,b,c,w\}$ induces a claw, and because $dbzcd$ and $dcvwd$ are 4-cycles that are not chorded, it follows that $bw$ is an edge.
Similarly, to prevent $\{d,a,b,c\}$ from inducing a claw, we must have the edge $ac$.
But now every vertex of $G$ has 4 neighbors, so $F=V(G)$.
Thus $n=9$, which is a contradiction.

\begin{figure}[h]
\begin{minipage}{.5\textwidth}
\centering
\includegraphics[width=3.2cm]{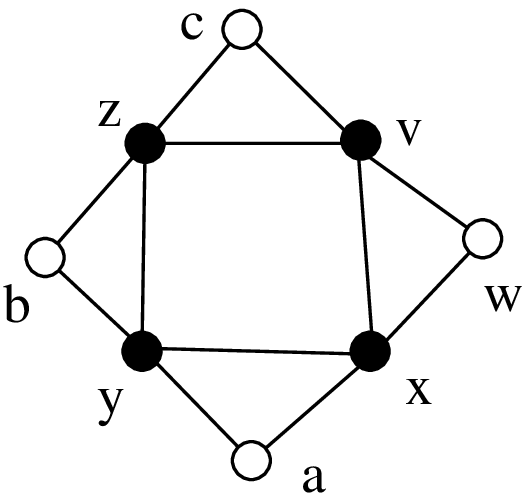}
\caption{For Subcase 2.1 of Lemma \ref{lem:P6C4*}.}
\label{fig:4flower}
\end{minipage}%
\begin{minipage}{.5\textwidth}
\centering
\vspace{0.29cm}
\includegraphics[width=3.1cm]{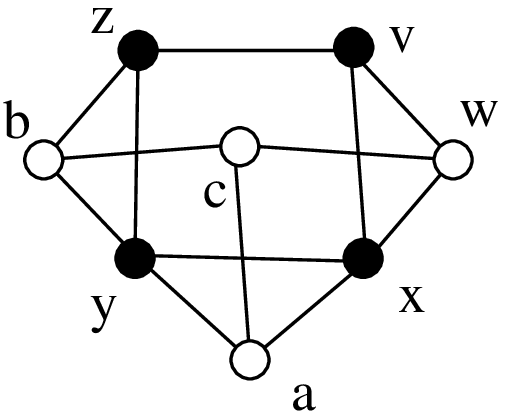}
\vspace{0.14cm}
\caption{For Subcase 2.2 of Lemma \ref{lem:P6C4*}.}
\label{fig:3flower}
\end{minipage}
\end{figure}

\noindent{\it Subcase 2.2.} Suppose that both $z$ and $v$ can be added to $F$.

The current state of $G$ (except for the vertex $c$, which we describe next) is given in Figure $\ref{fig:3flower}$.
Because $G$ is 2-connected, there is a $k$-tab with endpoints in the set $\{b,a,w\}$.
Thus, without loss of generality, we may assume that $cb$ is an edge.
To avoid $cbzvxa$ inducing a $P_6$, either $ba$ or $ca$ must be an edge.  
However, $ba$ gives a chorded $C_4$ so we may assume that $ca$ is an edge.
Because $cayzvw$ cannot be an induced $P_6$, and $aw$ creates a chorded $C_4$, it follows that $cw$ must also be an edge (see 
~{Figure \ref{fig:3flower}}).

But now $\{c,a,b,w\}$ cannot induce a claw, so $G$ contains at least one of $\{ab, aw, bw\}$.
Byecause $cbyac$ and $caxwc$ are 4-cycles that are not chorded, it follows that $bw$ is an edge.
Since $\Delta(G)\leq4$, we add $b$ and $w$ to $F$.
The only unfrozen vertices are $c$ and $a$.
If $c$ has a neighbor $d$ then $dcbzvx$ is an induced $P_6$, and if $a$ has a neighbor $d$ then $dayzvw$ is an induced $P_6$.
Therefore it follows that $n=8$, which is a contradiction.
\end{proof}


\begin{lem}\label{lem:P6C6*}
Let $G$ be a 2-connected graph of order $n\geq13$.  If $G$ is $\{K_{1,3},P_6\}$-free and $G$ contains a chorded $C_5$, then it also contains a chorded $C_6$.
\end{lem}
\begin{proof}
Let $G$ be a 2-connected graph of order $n\geq13$ that is $\{K_{1,3},P_6\}$-free, and let $C=vwxyzv$ be a $C_5$ with chord $vx$.
Suppose, for the sake of contradiction, that $G$ does not contain a chorded $C_6$.
Because $G$ is 2-connected and $n>5$, there is a $k$-tab on the cycle $C$.  
We choose a tab $Q=a_0a_1\cdots a_{k+1}$ that minimizes $k$.
Over all minimal $k$-tabs, choose one which minimizes the distance from $a_0$ to $a_{k+1}$ on $C$ . 
If $k\geq4$ then let $a_0'$ be a neighbor of $a_0$ on $C$ which is not $a_{k+1}$; now $a_0'a_0a_1\cdots a_4$ is an induced $P_6$.  Therefore $1\leq k\leq3$.
\vspace{.1in}

\noindent{\it Case 1.} Suppose that $Q$ is a 1-tab.

If $a_0a_{2}$ is an edge of $C$, then this edge is the chord of a 6-cycle.  
Therefore $a_0$ and $a_{2}$ are not neighbors on $C$.
If $a_0=w$ then by symmetry we may assume that $a_{2}=y$.  
Now $wa_1yzvxw$ is a $C_6$ with chord $vw$.
Therefore, we may assume without loss of generality that $a_0=x$ and $a_{2}=z$.  
By minimality of $Q$, neither of $va_1$ and $ya_1$ can be edges.  
Therefore $vy$ is an edge to avoid $\{z,a_1,v,y\}$ inducing a claw.
But now $vwxa_1zyv$ is a $C_6$ with chord $vz$.
\vspace{.1in}

\noindent{\it Case 2.} Suppose that $Q$ is a 2-tab.

We first consider when one endpoint of $Q$, say $a_0$ is $y$ (or $z$ by symmetry). 
 Because $\{y,a_1,z,x\}$ cannot induce a claw, $zx$ is an edge by minimality of $Q$. 
 If $a_{3}=x$ or if $a_{3}=z$ then we have a 6-cycle on the vertices $\{y,a_1,a_2,x,v,z\}$ with chord $zx$.
If $a_{3}=w$ then $ya_1a_2wvxy$ is a 6-cycle with chord $wx$, and if $a_{3}=v$ then $ya_1a_2vwxy$ is a 6-cycle with chord $vx$.
Therefore both $a_0$ and $a_{3}$ must be in $\{v,w,x\}$.

We may assume without loss of generality that $a_0$ is $x$ (or $v$ by symmetry). 
Now if $a_{3}=v$ then $xa_1a_2vzyx$ is a 6-cycle with chord $vx$, so it follows that $a_{3}=w$.
Because $\{x,y,v,a_1\}$ cannot induce a claw in $G$, and both $ya_1$ and $va_1$ create 1-tabs on $C$, it follows that $vy$ must be an edge.
But now $xa_1a_2wvyx$ is a 6-cycle with chord $vx$.
\vspace{.1in}

\noindent{\it Case 3.} Suppose that $Q$ is a 3-tab. 

Let $R=\{v,w,x,y,z,a_1,a_2,a_3\}$.
We first show that $a_0$ and $a_{4}$ cannot be neighbors in $C$.  
Suppose that $a_0$ and $a_{4}$ are neighbors in $C$. 
If $a_0=y$ then to avoid a claw at $y$ or a 1-tab on $C$, $zx$ must be an edge.  
But now the edges $\{yx,zy,xz\}$ along with $Q$ form a chorded 6-cycle.
The same contradiction arises if $a_0=z$.  
So if $a_0$ and $a_4$ are neighbors in $C$ then they both come from the set $\{v,w,x\}$. 
But now the edges $\{vw,wx,xv\}$ along with $Q$ form a chorded 6-cycle.  
Therefore we may assume that $a_0$ and $a_4$ are not neighbors in $C$.

Up to symmetry, the 3-tab $Q$ has $a_4=y$ and either $a_0=w$ or $a_0=v$.
Note that any other edge between $C$ and $\{a_1,a_2,a_3\}$ will create a 1- or 2-tab and contradict minimality.
Suppose first that $a_0=w$.
To avoid an induced claw centered at $y$ we must have the edge $xz$.
Also, $vwa_1a_2a_3y$ cannot be an induced $P_6$ so there must be another edge among these vertices.
Because $wy$ creates a chorded 6-cycle $wa_1a_2a_3yxw$, we must have the edge $vy$.
As $wy$ is not an edge and $wa_1a_2a_3yz$ cannot be an induced $P_6$, we also have the edge $zw$.
Now $G[\{v,w,x,y,z\}]=K_5-wy$ where $w$ and $y$ are the endpoints of $Q$ (see Figure \ref{fig:SameSame}).
If $R$ induces any other edge in $G$, then there is a chorded 6-cycle.

We suppose instead that $a_0=v$ and show that this gives an isomorphic graph to the graph in Figure \ref{fig:SameSame}.
Note that the edge $yv$ would create a 6-cycle $yxva_1a_2a_3y$ with chord $yv$.
To avoid 1-tabs or an induced claw on $\{y,a_3,x,z\}$ we must have the edge $xz$.
To avoid $wva_1a_2a_3y$ being an induced $P_6$, we must have the edge $wy$ (other edges would contradict the minimality of $Q$).
Further, $zw$ is an edge because $\{v,a_1,z,w\}$ cannot be an induced claw.
Now $G[V(C)]=K_5-e$, so we may continue using the notation used in Figure \ref{fig:SameSame}.

\begin{figure}[h]
\centering
\includegraphics[width=1.7in]{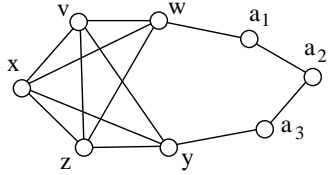}
\caption{The graph $G[R]$ in Case 3 of Lemma \ref{lem:P6C6*}. Note that $G[V(C)]=K_5-wy$.}
\label{fig:SameSame}
\end{figure}

We claim that $x$ can be added to $F$ (and $v$ and $z$ by symmetry).
Suppose instead that there is a vertex $b\in V(G)-R$ with edge $bx$.
Because $\{x,b,w,y\}$ cannot be an induced claw, either $bw$ or $by$ is an edge.
In the first case $xbwzvyx$ is a 6-cycle with chord $xw$, and in the second case $xbyvzwx$ is a 6-cycle with chord $xy$.
Thus, $x,v,z$ are all added to $F$.
\vspace{.1in}

\noindent{\it Subcase 3.1.} Suppose that $w$ and $y$ can be added to $F$.

Now $F=\{x,v,z,w,y\}$.
There is a tab on $G[R]$ because $G$ is 2-connected and $n>8$.
Thus, without loss of generality, there is a vertex $b$ such that $ba_1$ (or $ba_3$ by symmetry) is an edge.  
To avoid an induced claw $\{a_1,w,a_2,b\}$, we must also have the edge $ba_2$.  
Because $ba_2a_3yxw$ cannot be an induced $P_6$ and $y,x,w\in F$, we also have the edge $ba_3$.
Because there is no 2-tab on $C$, $G[R\cup\{b\}]$ induces no other edges.
However $n>9$, so there must be another vertex $c$.  
If $a_1$ and $a_3$ can now be added to $F$, then $\{a_1,a_3,c\}$ and their common neighbor induce a claw.
So we may assume that $ca_1$ is an edge (or $ca_3$ by symmetry).
To avoid  inducing claws with $\{a_1,c,w,a_2\}$ or $\{a_1,c,w,b\}$ both $ca_2$ and $cb$ must be edges.  
Because $cba_3yxw$ cannot be an induced $P_6$, $ca_3$ must also be an edge.  
Now $G[\{a_1,a_2,a_3,b,c\}]=K_5-e$ and by symmetry, $F=V(G)$.
Therefore $n=10$, which is a contradiction.
\vspace{.1in}

\noindent{\it Subcase 3.2.} Suppose that at least one of $w$ and $y$ cannot be added to $F$.

Recall that the current state of $G$ is described in Figure \ref{fig:SameSame} and that $F=\{v,x,z\}$.
We may assume without loss of generality that $bw$ is an edge in $G$ for $b\in V(G)-R$.
To avoid $\{w,b,v,a_1\}$ inducing a claw, $ba_1$ must also be an edge.
Now $ba_1a_2a_3yz$ is an induced $P_6$ unless $b$ has another neighbor in $R$.
Because $by$ and $ba_3$ create 1- and 2-tabs respectively, we must have the edge $ba_2$.
Note that $b$ has no other neighbors in $R$.

We claim that, if $b$ has a neighbor $b'\in V(G)-R$, then $N_G[b]=N_G[b']$.
Suppose that $b$ has a neighbor $b'\in V(G)-R$.
Note that $b'y$ makes $wbb'y$ a 2-tab and that $b'a_3$ gives a 6-cycle $b'a_3a_2a_1wbb'$ with chord $a_1b$, so neither of these can be edges.
Thus $b'bwxya_3$ is an induced $P_6$ unless $b'w$ is an edge.
But now $\{w,v,a_1,b'\}$ is an induced claw unless $b'a_1$ is an edge, and also $b'wxya_3a_2$ is an induced $P_6$ unless $b'a_2$ is an edge.
Therefore $N_G[b]=N_G[b']$.

If $\deg_G(b)=3$ then $b$ (and $a_1$ by symmetry) are now added to $F$.
If $b$ has a neighbor $b'$ in $V(G)-R$ then $G[w,b,b',a_1,a_2]=K_5-wa_2$.
Thus, by symmetry with $G[V(C)]$ we conclude that $b,b',$ and $a_1$ are all added to $F$ now.
Whether or not $b'$ exists, call $R'=R\cup N_G[b]$.
If $w$ has a neighbor $c\in V(G)-R'$ then $\{w,c,v,b\}$ is an induced claw because both $v$ and $b$ are in $F$.
Therefore $w$ is added to $F$ now and the graph $G[R']$ is shown in Figure \ref{fig:C6Second}, where the solid vertices (and $b'$, if it exists) are all in $F$.

\begin{figure}[h]
\centering
\includegraphics[width=1.6in]{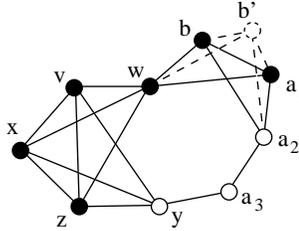}
\caption{This is $G[R']$.  Only $y,a_2,$ or $a_3$ can have neighbors in $G-R'$.}
\label{fig:C6Second}
\end{figure}

Because $n>10$  there must be a vertex $c\in V(G)-R'$.  
The only unfrozen vertices are $a_2,a_3$ and $y$.
If $ca_2$ is not an edge, then $cy$ and $ca_3$ must both be edges because $G$ is claw-free.  
But then $cyxwba_2$ would be an induced $P_6$.  
Therefore, $ca_2$ must be an edge.
To avoid $\{a_2,c,a_3,b\}$ inducing a claw, $ca_3$ is also an edge.  
Also $cy$ is an edge, because $ca_2bwxy$ cannot be an induced $P_6$.

Suppose that $c$ can be added to $F$ (and $a_3$ as well, by symmetry).
Because $n>11$ there is another vertex $d$ in $G$.
However the edge $yd$ makes $\{y,d,z,c\}$ induce a claw, and the edge $ya_2$ makes $\{a_2,d,a_1,c\}$ induce a claw.
Therefore $c$ cannot be added to $F$, and must have a neighbor $c'\in V(G)-R'$.
Because $c'cyxwb$ and $c'ca_2bwx$ cannot be induced $P_6$ subgraphs, both $c'y$ and $c'a_2$ must be edges.
Now $\{y,c',z,a_3\}$ cannot induce a claw, so $c'a_3$ is also an edge.
However, we now have $G[y,c,c',a_3,a_2]=K_5-ya_2$ and by symmetry $V(G)=F$.
It follows that $n=12$ (or $=11$ if $b'$ does not exist).
This is a contradiction because $n\ge13$.
\end{proof}

\begin{theorem}\label{thm:p6cpan} 
Let $G$ be a 2-connected graph of order $n\geq 13$.  If $G$ is $\{K_{1,3}, P_6\}$-free then $G$ is chorded pancyclic. 
\end{theorem}
\begin{proof}
By Theorem \ref{thm:p6}, we know that $G$ is pancyclic.  
Then by applying Lemma \ref{lem:P6C5*}, followed by Lemmas \ref{lem:P6C4*} and \ref{lem:P6C6*}, we find a chorded $m$-cycle in $G$ for $4\leq m\leq6$.
Any chordless $m$-cycle for $m>6$ contains an induced $P_6$.  
Therefore $G$ contains a chorded $m$-cycle for $4\leq m\leq n$
\end{proof}

Figure \ref{fig:P6Sharp} shows a 12-vertex, 2-connected, claw-free, and $P_6$-free graph which is not chorded pancyclic, because there is no chorded $6$-cycle.  This proves that Theorem \ref{thm:p6cpan} is sharp.

\begin{figure}[h]
\centering
\includegraphics[width=4cm]{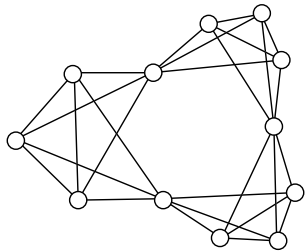}
\caption{This graph has no chorded $C_6$, which shows that Theorem \ref{thm:p6cpan} is sharp.}
\label{fig:P6Sharp}
\end{figure}

\end{document}